\documentclass[12pt]{amsart}

\usepackage{amsfonts}
\usepackage{amscd}
\usepackage{amssymb}

\allowdisplaybreaks
\usepackage{graphicx}
\usepackage{xcolor}
\usepackage{comment}
\usepackage{array}
\usepackage{tikz-cd}
\usepackage{cite}
\usepackage{enumerate,enumitem,hyperref,cleveref}
\hypersetup{colorlinks,linkcolor={blue},citecolor={blue},urlcolor={blue}}

\newtheorem{thm}{Theorem}[section]

\newtheorem{lem}[thm]{Lemma}

\newtheorem{prop}[thm]{Proposition}

\theoremstyle{remark}
\newtheorem{rem}[thm]{Remark}

\numberwithin{equation}{section}

\usepackage{color}

\title[Fractional Dunkl heat equation]
{Large time behavior of the fractional heat equation associated with the Dunkl Laplacian}

\author[S. Mukherjee]{Suman Mukherjee}

\address{Department of Mathematics, Indian Institute of Technology Bombay, Powai, Mumbai--400076, India.}
\email{sumanmukherjee822@gmail.com}


\begin{document}

\begin{abstract} 
We consider the fractional heat equation associated with the Dunkl Laplacian and prove that the weak solutions to this equation converge to the fundamental solution as time becomes large, provided the initial data is an integrable function with respect to the associated measure. As an application, we also prove a similar result for the corresponding nonlinear equation.
\end{abstract}
  \maketitle

\section{Introduction} \label{Sec-intro} 

In \cite{VazquezABFTFHEITES, VazquezABMFTHE}, V\'{a}zquez studied fractional heat equations of the form  
$$\partial_t u + (-\Delta)^{\alpha} u = 0,$$
where $(-\Delta)^{\alpha}$ denotes the fractional Laplacian of order $\alpha$ and $0 < \alpha \leq 1$. He established the asymptotic convergence of weak solutions to the corresponding fundamental solution. These results have important implications for the study of heat equations of both integer and fractional order. In particular, the papers themselves use these findings to study the Fokker--Planck and the Ornstein--Uhlenbeck equations. Moreover, these long-time asymptotic convergence results are closely connected to the `Central Limit Theorem' of probability in the PDE setting; see \cite{VazquezABMFTHE} for further details.

\par For these reasons, the asymptotic behavior of heat equations has recently become an important focus in analysis, and these results have been generalized to several abstract settings. V\'{a}zquez \cite{VazquezABFTHEOHS} himself considered the asymptotic behavior of the heat equation in multidimensional hyperbolic space for integrable data. This line of research was later extended by Anker, Papageorgiou, and Zhang \cite{AnkerABOSTTHEONSS}, who investigated the asymptotic behavior of solutions to the heat equation on Riemannian symmetric spaces of noncompact type. Papageorgiou \cite{PapageorgiouABOSTTEPFTFLONSS} studied the asymptotic behavior of solutions to the extension problem for the fractional Laplacian on noncompact symmetric spaces. In another work \cite{PapageorgiouLTBOTFOORTTFLCRM}, she examined the large-time behavior of two families of operators: one arising from the Caffarelli--Silvestre extension and the other from the fractional heat equation on certain Riemannian manifolds. More recently, Naik, Ray, and Sarkar \cite{NaikLABOSOFTFHEORSSONT} investigated the asymptotic behavior of solutions to the heat equation and its fractional counterpart on Riemannian symmetric spaces of noncompact type.

\par With this background, we now turn our attention in a different direction. Dunkl theory, introduced in 1989 through the pioneering work of Charles F. Dunkl \cite{Dunkl}, represents a deformation of Euclidean Fourier analysis that has attracted significant attention from researchers. At the core of this theory lie the Dunkl operators, which serve as analogs to classical partial derivatives in the Dunkl framework. These operators fulfill a dual purpose: they provide algebraic tools for exploring special functions and establish the analytical foundation of the Dunkl theory.
\par One of the fundamental operators in this setting is the Dunkl Laplacian $\Delta_k$, given by 
$$\Delta_k f(x):= \sum\limits_{j=1}^d T_j^2 f(x)= \Delta f(x)
+ \sum_{\lambda\in R} k(\lambda)
\left(
\frac{\langle\nabla f(x),\lambda\rangle}{\langle\lambda,x\rangle}
-\frac{|\lambda |^2}{2}
\frac{f(x)-f(\sigma_\lambda x)}{\langle\lambda,x\rangle^2}
\right),$$ 
where the notation is explained in Section \ref{Sec-prel}. This operator serves as a central tool for the analysis of semigroups and partial differential equations in the Dunkl framework. For instance, R\"osler introduced \cite{RoslerGHPATHEFDO} the following initial value problem for the heat equation associated with the Dunkl Laplacian
\begin{align*}
    \begin{cases}
       \partial_t u(x,t) -\Delta_k \,u(x,t) &= 0,\, \quad (x,t)\in \mathbb{R}^d \times (0,\infty),\\
       \hfill u(x,0) &= u_0(x).
    \end{cases}
\end{align*}
The solution to this problem can be expressed as the Dunkl convolution of the Dunkl heat kernel $h_t$ with the initial data $u_0$. The Dunkl heat kernel and the Dunkl heat semigroup have been extensively studied in the literature for its applications in harmonic analysis and PDEs in the Dunkl setting. For details, we refer to the references \cite{HejnaHMT, HejnaHFCHF, HejnaUALBFLPSFIDS, HejnaUALBFTDHK, HejnaHSFTDHO, HejnaROADFTHSHITRDS, HejnaRODTONRK, HanRTACITDS, HanLATLSCIDS, LiDFSFEFDO, SumanDPAFLRFTDL}.
\par More generally, for $0 < \alpha \leq 1$, one can consider the fractional Dunkl heat equation
\begin{align}\label{linear frac Dunkl heat eq}
    \begin{cases}
       \partial_t u(x,t) + (-\Delta_k)^{\alpha} \,u(x,t) &= 0,\, \quad (x,t)\in \mathbb{R}^d \times (0,\infty),\\
       \hfill u(x,0) &= u_0(x).
    \end{cases}
    \end{align}
Similar to the previous case, its solution can be expressed as the Dunkl convolution of the fractional Dunkl heat kernel $h_{t,\alpha}$ with the initial value $u_0$. The fractional Dunkl heat semigroup has not been extensively studied in the literature; however, it has recently attracted significant attention and is expected to receive even greater focus in the future. For instance, as demonstrated by the author of this article and Thangavelu \cite{SumanOTBODM}, this fractional semigroup offers a powerful approach for establishing sharp multiplier theorems for the Dunkl transform.

\par Recalling the discussion at the beginning of the paper, it is therefore natural to ask whether similar asymptotic behavior can be studied in the Dunkl context. In this paper, we aim to advance the theory of PDEs associated with Dunkl operators by investigating the large-time behavior of solutions to the fractional Dunkl heat equation \eqref{linear frac Dunkl heat eq}. We study weak solutions of this equation and show that, as $t \to \infty$, they converge to the fundamental solution, provided that the initial data $u_0$ is an integrable function. Formally, our main theorem can be stated as follows.

\begin{thm}\label{main thm linear part}
Let $1\leq p\leq \infty$, and denote by $p'$ the conjugate exponent of $p$ given by $1/p+1/p'=1$, and $u_0 \in L^1(d\mu_k)$ be such that 
$$\int_{\mathbb{R}^d}u_0(x)\,d\mu_k(x)=M\neq 0.$$
Then
\begin{equation*}\label{linear lp assymp behav}
\lim\limits_{t\to \infty} t^{d_k/(2\,\alpha\,p')} \|h_{t,\alpha}\ast_k u_0 - Mh_{t,\alpha}\|_{L^p(d\mu_k)}  =0.
\end{equation*}
\end{thm}
\begin{rem}
In the endpoint cases, when $p=\infty$, the scaling factor
$t^{d_k/(2\alpha p')}$ reduces to $t^{d_k/(2\alpha)}$, whereas when
$p=1$, the factor disappears.
\end{rem}
\par Our proof is grounded in an analysis of the error via a representation formula analogous to the classical case; nevertheless, several significant differences emerge. For instance, the classical fractional heat kernel possesses a well-known self-similar formula. By utilizing this formula, along with certain established pointwise estimates \cite{PruittTPKAHPFTGSP} of the fractional heat kernel and its derivatives, key results were obtained in \cite{VazquezABFTFHEITES}. Additionally, the difference between fractional heat kernels was estimated through the mean value theorem, where the translation invariance of the Lebesgue measure played a pivotal role. In our case, while a self-similar formula can be derived, no comparable pointwise estimates are available for the fractional Dunkl heat kernel. Moreover, the difference between fractional heat kernels arises as the difference of Dunkl translations of fractional heat kernels. Consequently, these tools cannot be applied to our scenario. The key element in this article lies in selecting an effective method for the proof that is appropriately tailored to the Dunkl context. Here, we first rely on the known estimates of the Dunkl heat kernel and its Dunkl translations to establish the main theorem for the heat equation, that is, for the case $\alpha=1$. Subsequently, by applying Bochner's subordination formula and carefully employing certain properties of the subordinate function and Dunkl convolution, we provide the complete proof.
\par Theorem \ref{main thm linear part} suggests several further directions for study. For instance, the Caffarelli--Silvestre extension problem associated with the fractional Dunkl Laplacian was investigated in \cite{AnoopHIATHIFDG}. Our result therefore motivates the study of the large-time behavior of such extension problems in the Dunkl setting, in analogy with the corresponding developments on noncompact symmetric spaces and Riemannian manifolds \cite{PapageorgiouABOSTTEPFTFLONSS, PapageorgiouLTBOTFOORTTFLCRM}.

\par In this article, we also present one application. Specifically, as an application of Theorem \ref{main thm linear part}, we examine the large-time behavior of the following nonlinear fractional Dunkl heat equation.
\begin{align}\label{non-linear frac Dunkl heat eq}
    \begin{cases}
       \partial_t u(x,t) + (-\Delta_k)^{\alpha} \,u(x,t) &= - u^m(x,t),\, \quad (x,t)\in \mathbb{R}^d \times (0,\infty),\\
       \hfill u(x,0) &= u_0(x).
    \end{cases}
    \end{align}
Using Theorem \ref{main thm linear part} and following the method of Fino and Karch \cite{FinoDOMFNEWFL}, we obtain the following asymptotic behavior for the solutions of the above equation.
\begin{thm}\label{main thm non-linear part}
Let $m>1+2\alpha/d_k$ and let
$0\le u_0\in L^1(d\mu_k)$, with $u_0\not\equiv0$.
Let $u$ be the unique global nonnegative mild solution
of \eqref{non-linear frac Dunkl heat eq} and 
$$M(t)=\int_{\mathbb{R}^d}u(x, t)\,d\mu_k(x).$$
Then
$$\lim\limits_{t\to \infty}\int_{\mathbb{R}^d}u(x, t)\,d\mu_k(x) =\lim\limits_{t\to \infty} M(t)=M_\infty >0.$$
Also,
\begin{equation*}\label{non-linear lp assymp behav}
\lim\limits_{t\to \infty} t^{d_k/(2\,\alpha\,q')} \|u(\cdot, t) - M_\infty h_{t,\alpha}\|_{L^q(d\mu_k)}  =0, \text{ for } 1\leq q <\infty.
\end{equation*}
\end{thm}
\begin{rem}
Due to the limited understanding of fractional Dunkl derivatives and the lack of lower bounds for Dunkl translations, we are unable to address the case $m\leq 1+{2\alpha}/{d_k}$ here. Investigating this direction would be an interesting problem.
\end{rem}
\par We conclude this section with an outline of the structure of the paper. In Section \ref{Sec-prel}, we present the preliminary results and notations required for reading the paper. In Section \ref{Sec-frac Dunkl heat ker}, we state some known results regarding the Dunkl heat kernel and Dunkl heat semigroup. We also introduce the fractional Dunkl heat semigroup in this section and prove some properties of the fractional Dunkl heat kernel. Section \ref{Sec- linear Main Thm proof} is devoted to the proof of Theorem \ref{main thm linear part}. Section \ref{Sec- Nonlinear Main Thm proof} begins with global well-posedness and a comparison principle for the nonlinear equation \eqref{non-linear frac Dunkl heat eq}. These results are then used to prove Theorem \ref{main thm non-linear part}.

\section{Preliminaries and notation
}\label{Sec-prel}
Dunkl theory is now familiar to experts in the field. However, to avoid confusion for readers unfamiliar with Dunkl theory, we will outline the necessary preliminaries and notation for this article.
\subsection{The Dunkl setup}
\par Let $\langle \cdot, \cdot \rangle$ denote the usual dot product in $\mathbb{R}^d$, and let $|\cdot| = \sqrt{\langle \cdot, \cdot \rangle}$ denote the usual norm. For any $\lambda \in \mathbb{R}^d$ with $|\lambda|=\sqrt{2}$, we define $\sigma_\lambda: \mathbb{R}^d \to \mathbb{R}^d$ by 
$$\sigma_\lambda (x) = x - \langle x, \lambda \rangle \lambda.$$
The map $\sigma_\lambda$ is known as the \emph{reflection} associated with the vector $\lambda$.
In this article, we denote by $R$ a \emph{normalized root system}, that is, $R \subseteq \mathbb{R}^d \setminus \{0\}$ with the properties
\begin{enumerate}[label=(\roman*)]
    \item $R$ is finite.
    \item $R \cap \mathbb{R}\lambda=\{\lambda, -\lambda\}$, for any $\lambda \in R$.
    \item $\sigma_\lambda(R)=R,$ for any $\lambda \in R$.
    \item $|\lambda|^2=2$, for any $\lambda \in R$.
\end{enumerate}
The reflections $\{\sigma_\lambda : \lambda \in R\}$ generate a finite group $G \subseteq O(d, \mathbb{R})$, which is known as the \emph{reflection group} associated with the root system $R$.
\par Due to the action of the reflection group $G$, a new `metric' on $\mathbb{R}^d$ can be introduced, defined by $d_G(x, y) = \min |\sigma(x) - y|$, for any $x, y \in \mathbb{R}^d$. It is obvious that $d_G(x, y) = 0$ does not necessarily imply $x = y$. However, it satisfies the other properties of a metric, namely nonnegativity, symmetry, and the triangle inequality.

\par We define a function $k : R\to [0, \infty)$ such that $k$ is $G$-invariant, that is, $k(\sigma(\lambda))= k(\lambda)$, for any $\sigma \in G$ and for any $\lambda \in R$. The function $k$ is known as the \emph{multiplicity function}. 
\par  Let $v_k$ be the $G$-invariant weight function given by $v_k(x) = \prod _ {\lambda \in R} | \langle x , \lambda \rangle|^{k (\lambda)}$. Clearly, $v_k$ is a homogeneous function of degree $\gamma_k := \sum_{\lambda \in R} k(\lambda)$. Let $d\mu_k(x)$ be the associated measure, given by $c_k \, v_k(x) \, dx$, where
$$c_k^{-1} = \int_{\mathbb{R}^d} e^{-\frac{|x|^2}{2}} \, v_k(x) \, dx.$$
It is also clear from the homogeneities of the function $v_k$ and the Lebesgue measure that $d\mu_k$ is homogeneous of degree $d_k := d + \gamma_k$. One of the main difficulties that occur when working with this measure is that the volume of a ball not only depends on the radius but also on the center. To be more precise,
\begin{equation*}\label{Volofball}
    \mu_k(B(x,r))\approx \ r^d\prod\limits_{\lambda \in R}\left(|\langle x,\lambda \rangle|+r\right)^{k(\lambda)}.
\end{equation*}

\subsection{Dunkl operators, Dunkl kernel and Dunkl transform}
In 1989, Charles F. Dunkl \cite{Dunkl} introduced the \emph{Dunkl operators} $\{T_j: 1\leq j\leq d\}$, given by 
\begin{eqnarray*}
T_j f(x)= \partial_j f(x)+\sum\limits_{\lambda \in R} \frac{k(\lambda )}{2} \lambda _j \frac{f(x) - f(\sigma _\lambda  x)}{\langle \lambda  , x \rangle}, \quad 1\leq j \leq d;
\end{eqnarray*}
 where $\lambda =(\lambda_1,\lambda_2,...,\lambda_d)$ and $f\in C^1(\mathbb{R}^d)$.
These operators can be considered a generalization of the partial derivative operators and coincide with them when the multiplicity function $k$ is zero.
\par Given any $y \in \mathbb{R}^d$, it is known (see \cite{DunklIKWRGI}) that there exists a unique real-analytic solution $f(x) = E_k(x, y)$ to the system of equations.
 \begin{align*}
    \begin{cases}
       T_j f  &=  y_j f, \hspace{.5cm}1\leq j\leq d;\\
       \hfill f(0) &= 1.
    \end{cases}
\end{align*}
We call $E_k$ the \emph{Dunkl kernel}, and it serves the role of the exponential function in this setup. In fact, it shares several properties with the exponential function, which are as follows
\begin{enumerate}[label=(\roman*)]
\item $E_k$ extends to a unique holomorphic function on $\mathbb{C}^d \times \mathbb{C}^d$.
    \item $E_k(x,0) = 1$ for any $x\in \mathbb{C}^d$.
    \item $E_k(x,y) = E_k(y,x)$ for any $x, y \in \mathbb{C}^d$.
    \item $E_k(tx,y) = E_k(x,ty)$ for any $x, y \in \mathbb{C}^d$ and for any $t \in \mathbb{C}$. 
    \item $|E_k(ix,y)| \leq 1,$ for any $x, y\in \mathbb{R}^d$.
\end{enumerate}
These results are well known in the literature, and their proofs can be found in \cite{DunklIKWRGI, RoslerDOTA}.

\par For $1 \leq p < \infty$, let $L^p(d\mu_k)$ denote the Lebesgue spaces over $\mathbb{R}^d$ with respect to the measure $d\mu_k$, and let $L^\infty(d\mu_k)$ denote the usual $L^\infty$ space over $\mathbb{R}^d$.
Similar to the Fourier transform, the \emph{Dunkl transform} of a function $f$ is given by:
\begin{eqnarray*}
\mathcal{F}_kf(\xi)=c_k\,\int_{\mathbb{R}^d}f(x)E_k(-i\xi,x)\,v_k(x)\,dx =\int_{\mathbb{R}^d}f(x)E_k(-i\xi,x)\,d\mu_k(x).
\end{eqnarray*}
From the properties of the Dunkl kernel, the above definition makes sense for any function $f \in L^1(d\mu_k)$. Moreover, in the case $k \equiv 0$, $\mathcal{F}_k$ reduces to the classical Fourier transform, and thus the Dunkl transform serves as a generalization of the Fourier transform. We now list a few useful properties of $\mathcal{F}_k$ that are similar to those of the Fourier transform
(for details, we refer the reader \cite{deJeuTDT, RoslerDOTA, XuBook}).
\begin{enumerate}[label=(\roman*)]
    \item The inversion formula $$f(x)=\mathcal{F}_k^{-1}(\mathcal{F}_kf)(x)=\int_{\mathbb{R}^d}\mathcal{F}_kf(\xi)E_k(i\xi,x)\,d\mu_k(\xi),$$ holds for any function $f$ such that $f, \, \mathcal{F}_kf \in L^1(d\mu_k)$.
    \item The operator $\mathcal{F}_k$ is an isometry on $L^2(d\mu_k)$, that is, $||\mathcal{F}_kf||_{L^2(d\mu_k)}=||f||_{L^2(d\mu_k)}$.
    \item  $\mathcal{F}_k (f_t)(\xi)= \mathcal{F}_k\,f (t\xi)$,  where  $f_t(x) := t^{-d_k} f(x/t), t>0$.
    \item If $f$ is radial, then $\mathcal{F}_k f$ is also radial.
\end{enumerate}

\subsection{Dunkl translation and Dunkl convolution}
As the associated measure is not invariant under usual translation, a new operator, called the \emph{Dunkl translation}, was introduced to support the development of harmonic analysis in this framework. Formally, the Dunkl translation $\tau^k_x: L^2(d\mu_k) \to L^2(d\mu_k)$ is a bounded operator given by
$$\mathcal{F}_k(\tau^k_x f)(y) = E_k(ix,y) \mathcal{F}_k f(y).$$
Here, we mention a few well-established properties of the Dunkl translation operator, which can be found in \cite{ThangaveluCOMF}, for later use.

\begin{enumerate}[label=(\roman*)]
\item $\tau^k_0=I$, the identity operator.
\item \label{pointwise formula for Dunkl trans} 
The pointwise formula 
$$
\tau^k_yf(x)=\tau^k_x f(y) = \int_{\mathbb{R}^d} E_k(ix, \xi) E_k(iy, \xi) \mathcal{F}_k f(\xi) \, d\mu_k(\xi)
$$ 
holds for any function $f\in L^2(d\mu_k)$ such that $f, \, \mathcal{F}_kf \in L^1(d\mu_k)$.

\item \label{positivity of Dunkl trans} 
If $f$ is a reasonable radial function such that $f \geq 0$ a.e., then 
$\tau^k_x f \geq 0$ a.e..
\item \label{bddness of Dunkl trans on radial Lp}
 $\tau^k_x : L^p(d\mu_k) \to L^p(d\mu_k)$ is a bounded operator, $1 \leq p \leq \infty$, when restricted to radial functions (see \cite{GorbachevPLBDTGTOAIA, ThangaveluCOMF}). In general, the boundedness of $\tau^k_x$ is not known for all the $L^p$-functions.
\end{enumerate}
\par For $f,g \in L^2(d\mu_k)$, the \emph{Dunkl convolution} of $f$ and $g$ is defined by 
	$$f\ast_kg(x)=\int_{\mathbb{R}^d}f(y)\tau^k_xg(-y)\,d\mu_k(y).$$
	It is easy to see that the Dunkl convolution possesses the following properties:
		\begin{enumerate}[label=(\roman*)]
			\item $f\ast_kg(x)=g\ast_kf(x)$ for any $f, g \in  L^2( d\mu_k)$.
			
			\item\label{dunkl conv is commutative} $\mathcal{F}_k(f\ast_kg)(\xi)=\mathcal{F}_kf(\xi) \mathcal{F}_kg(\xi)$ for any $f, g \in  L^2(d\mu_k)$.
		\end{enumerate}
Due to the lack of boundedness of Dunkl translations in $L^p$ spaces, Young's inequality for Dunkl convolution is not available. However, if one of the functions is taken to be radial, then this inequality holds. Since we will be using it in the article many times, we state it formally below.
\begin{thm}\label{Youngs ineq} \cite[Theorem 3.6]{GorbachevPLBDTGTOAIA}
Let $1\leq p, q, r\leq \infty$ with
$1/p + 1/q=1+1/r$. Also, let $f \in L^p(d\mu_k)$ and $g \in L^q(d\mu_k)$ be such that one of them is radial.
 Then
$$\|f\ast_kg\|_{ L^r(d\mu_k)} \leq  \|f\|_{L^p(d\mu_k)}  \|g\|_{L^q(d\mu_k)}.$$
\end{thm}


\section{Dunkl heat kernel and fractional Dunkl heat kernel}\label{Sec-frac Dunkl heat ker} 
In this section, we begin by discussing several properties of the Dunkl heat kernel and the Dunkl heat semigroup, as developed by R\"osler \cite{RoslerGHPATHEFDO}.

\par Recall that $\Delta_k=\sum_{j=1}^d T_j^2$ denotes the \emph{Dunkl Laplacian}. Now the \emph{Dunkl heat semigroup} $\{e^{t\,\Delta_k}\}_{t\geq 0}$ is given by
$$e^{t\,\Delta_k}f(x)= \mathcal{F}_k^{-1}(e^{-t|\cdot|^2} \mathcal{F}_kf)(x).$$
In fact, $e^{t\,\Delta_k} f$ is a solution of the Dunkl heat equation
 \begin{align}\label{Dunkl heat eq}
    \begin{cases}
       \partial_t u(x,t) -\Delta_k \,u(x,t) &= 0,\, \quad (x,t)\in \mathbb{R}^d \times (0,\infty),\\
       \hfill u(x,0) &= f(x).
    \end{cases}
    \end{align}
By applying the Dunkl transform, it can be written in terms of Dunkl convolution as
$$e^{t\,\Delta_k}f(x)= f \ast_k h_t(x)= \int_{\mathbb{R}^d}\tau^k_xh_t(-y)f(y) \, d\mu_k(y),$$
where $\tau^k_xh_t(-y)$ is known as the \emph{Dunkl heat kernel} and is given by
$$h_t(x)=\mathcal{F}_k^{-1}(e^{-t|\cdot|^2})(x)= (2t)^{-d_k/2}e^{-|x|^2/(4t)}.$$
Thus, $h_t$ is a nonnegative Schwartz function, and it also has the nice property that for any $t>0$,
\begin{equation}\label{int of heat ker is 1}
  \int_{\mathbb{R}^d}h_t(x)\, d\mu_k(x)=1 .
\end{equation}
The following estimates of the Dunkl heat kernel obtained by Anker, Dziuba\'{n}ski and Hejna are very useful for us.
\begin{thm}\label{Dunkl heat ker esti} \cite[Theorem 4.1]{HejnaHFCHF}
 Let $t>0$  and $x, y\in \mathbb{R}^d$. Then
 \begin{enumerate}[label=(\roman*)]
     \item \leavevmode\vspace*{-\dimexpr\abovedisplayskip + \baselineskip}$$|\tau^k_x h_t(-y)| \lesssim \frac{1}{\mu_k(B(y, \sqrt{t}))}\, e^{{-c\,d_G(x,\, y)^2}/{t}}.$$
     
     \item \leavevmode\vspace*{-\dimexpr\abovedisplayskip + \baselineskip} $$|\tau^k_x h_t(-y)-\tau^k_x h_t(-y')| \lesssim \frac{|y-y'|}{\sqrt{t}}\frac{1}{\mu_k(B(y, \sqrt{t}))}\, e^{{-c\,d_G(x,\, y)^2}/{t}},$$
     if $|y-y'|<\sqrt{t}$.
 \end{enumerate}
\end{thm}
Next, we define the semigroup generated  by the fractional powers of $(-\Delta_k)$. For any $0<\alpha\leq 1$, let 
$$e^{-t\,(-\Delta_k)^{\alpha}}f(x):= \mathcal{F}_k^{-1}(e^{-t|\cdot|^{2\alpha}} \mathcal{F}_kf)(x).$$
When $\alpha=1$, the fractional Dunkl heat semigroup $\{e^{-t\,(-\Delta_k)^{\alpha}}\}_{t\geq 0}$ coincides with the Dunkl heat semigroup $\{e^{t\,\Delta_k}\}_{t\geq 0}$.
Since, by \cite{RoslerDOTA}, the Dunkl heat semigroup is a strongly continuous contraction semigroup on $L^p(d\mu_k)$, $1\leq p<\infty$, using Bochner’s subordination formula \cite{BochnerDEASP}, we can write
\begin{equation}\label{Bochner sub ord formula}
 e^{-t\,(-\Delta_k)^{\alpha}}f(x)= \int_0^\infty e^{s\,\Delta_k}f(x)\, \eta_{t,\alpha}(s) \, ds,  \quad 0<\alpha<1, 
\end{equation}
where $\eta_{t,\alpha}$ is a probability density function on $(0, \infty)$. According to \cite[p. 20]{GrigoryanHKAFTOMMS} (see also \cite{RejebSRRTTFDL}) and \cite[p. 89]{GraczykTDEFSPOSS}, $\eta_{t,\alpha}$ obeys the following properties
 \begin{enumerate}[label=(\roman*)]
      \item $\eta_{t,\alpha}(s) = t^{-1/\alpha}\,\eta_{1,\alpha}(s/t^{1/\alpha})$ for any  $t, s>0$.
      \item For any $\delta>0$, $\int_{0}^{\infty} s^{-\delta}\,\eta_{1,\alpha}(s) \,ds<\infty$.
      \item There exist $a_\alpha, b_\alpha \in \mathbb{R}$ and $c_\alpha>0$ such that
      $$\eta_{t,\alpha}(s) \lesssim t^{\,a_\alpha}(1+s^{b_\alpha})\,e^{-c_\alpha\, t^{\frac{1}{1-\alpha}}s^{-\frac{\alpha}{1-\alpha}}}$$ for any  $t>1, s>0$.
 \end{enumerate}
Now, similar to the case of the Dunkl heat semigroup, by applying the Dunkl transform, we can write
$$e^{-t\,(-\Delta_k)^{\alpha}}f(x)= f \ast_k h_{t,\alpha}(x)= \int_{\mathbb{R}^d}\tau^k_xh_{t,\alpha}(-y)f(y) \, d\mu_k(y),$$
where $\tau^k_x h_{t, \alpha}(-y)$ is called the \emph{fractional Dunkl heat kernel} and is given by
$$h_{t, \alpha}(x) = \mathcal{F}_k^{-1}(e^{-t|\cdot|^{2\alpha}})(x) .$$
Although we do not have an explicit expression for 
$h_{t, \alpha}$, it can  be written using the subordination formula as
\begin{equation}\label{eq:sub or for}
 h_{t, \alpha}(x)= \int_0^\infty h_s(x)\, \eta_{t,\alpha}(s) \, ds.    
\end{equation}
The above formula, together with \eqref{int of heat ker is 1} and the homogeneity of the measure $d\mu_k$, gives the estimate.
\begin{equation}\label{Lp norm of frac heat ker}
  \|h_{t, \alpha}\|_{L^p(d\mu_k)} \lesssim t^{-d_k/(2\,\alpha\,p')} , \text{ for any } 1\leq p \leq \infty. 
\end{equation}

\section{Proof of Theorem \ref{main thm linear part}}\label{Sec- linear Main Thm proof}
\subsection{Proof for the case $\alpha=1$}
We will first prove Theorem \ref{main thm linear part} for the case $\alpha = 1$ and then proceed with the complete proof. For this, let us define a quantity $N_1$, which is often called the first absolute moment, defined for any $f\in C_c^\infty(\mathbb{R}^d)$ by
$$N_1(f)= \int_{\mathbb{R}^d} |f(x)|\,|x|\, d\mu_k(x).$$
The following proposition plays a key role in our proof.
\begin{prop}\label{1st moment prop}
 Let $f \in L^1(d\mu_k)$  and $N_1(f)<\infty$. Then
 $$\|h_{t}\ast_k f - Mh_{t}\|_{L^1(d\mu_k)}  \lesssim N_1(f)\, t^{-1/2}$$
 $$\text{ and  } t^{d_k/2} \,\|h_{t}\ast_k f - Mh_{t}\|_{L^\infty(d\mu_k)}  \lesssim N_1(f)\, t^{-1/2}.$$
\end{prop}
\begin{proof}
 We begin by using the properties of Dunkl translation to write 
 \begin{eqnarray*}
&&h_{t}\ast_k f(x) - Mh_{t}(x) \\
&=& \int_{\mathbb{R}^d} f(y) [ \tau^k_x h_t(-y) - \tau^k_x h_t(0)] \,d\mu_k(y)\\
&=& \int_{|y|<\sqrt{t}} f(y) \big[ \tau^k_x h_t(-y) - \tau^k_x h_t(0)\big] \,d\mu_k(y) + \int_{|y| \geq \sqrt{t}} f(y) \big[ \tau^k_x h_t(-y) - \tau^k_x h_t(0)\big] \,d\mu_k(y)\\
&=:& A_t + B_t.
 \end{eqnarray*}
Now, using Theorem \ref{Dunkl heat ker esti} (ii), we obtain
\begin{eqnarray*}
|A_t| &\lesssim & \int_{|y|<\sqrt{t}} |f(y)| \frac{|y|}{\sqrt{t}}\, t^{-d_k/2} e^{-c|x|^2/t}\, d\mu_k(y)\\
&=& t^{-1/2}\int_{|y|<\sqrt{t}} |f(y)|\,|y|\, t^{-d_k/2} e^{-c|x|^2/t}\, d\mu_k(y).
\end{eqnarray*}

Similarly, using the expression for $h_t$, we can write
\begin{eqnarray*}
|B_t| &\leq & \int_{|y|\geq \sqrt{t}} |f(y)| \big[  \tau^k_x h_t(-y)+  h_t(x)\big]\, d\mu_k(y)\\
&\lesssim& \int_{|y|\geq \sqrt{t}} |f(y)|\, \big[ \tau^k_x h_t(-y) + t^{-d_k/2} e^{-|x|^2/(4t)} \big]\, d\mu_k(y)\\
&\leq & t^{-1/2}\int_{|y|\geq \sqrt{t}} |f(y)|\,|y|\, \big[ \tau^k_x h_t(-y) + t^{-d_k/2} e^{-|x|^2/(4t)} \big]\, d\mu_k(y).
\end{eqnarray*}
Thus, by applying the above estimates for $A_t$ and $B_t$, using the $L^1$-boundedness of the Dunkl translation operator on radial functions, along with the property  \eqref{int of heat ker is 1}, we finally get
\begin{eqnarray*}
 \|h_{t}\ast_k f - Mh_{t}\|_{L^1(d\mu_k)}  \lesssim t^{-1/2}\int_{\mathbb{R}^d} |f(y)|\,|y|\,d\mu_k(y)=N_1(f)\, t^{-1/2}.  
\end{eqnarray*}
This completes the proof of the first estimate.
\par For the second estimate, we note that the estimates for $A_t$ and $B_t$ can be refined as
\begin{eqnarray*}
|A_t| &\lesssim & \int_{|y|<\sqrt{t}} |f(y)| \frac{|y|}{\sqrt{t}}\, t^{-d_k/2} e^{-c|x|^2/t}\, d\mu_k(y)\\
&\leq & t^{-(d_k+1)/2}\int_{|y|<\sqrt{t}} |f(y)|\,|y|\,d\mu_k(y)
\end{eqnarray*}
and 
\begin{eqnarray*}
|B_t| &\lesssim &  \int_{|y|\geq \sqrt{t}} |f(y)| \big[ \frac{e^{-c\,d_G(x,y)^2/t}}{\mu_k(B(y, \sqrt{t}))}+ \frac{e^{-|x|^2/(4t)}}{t^{d_k/2}}\big]\, d\mu_k(y)\\
&\lesssim & \int_{|y|\geq \sqrt{t}} |f(y)| \frac{|y|}{\sqrt{t}}\big[ \frac{1}{t^{d_k/2}}+ \frac{1}{t^{d_k/2}}\big]\, d\mu_k(y)\\
&= & t^{-(d_k+1)/2}\int_{|y|\geq \sqrt{t}} |f(y)|\,|y|\,d\mu_k(y).
\end{eqnarray*}
Hence, the proof of the $L^\infty$ estimate follows in a similar manner.
\end{proof}
We are now prepared to establish Theorem \ref{main thm linear part} for $\alpha=1$.
\begin{proof}[Proof of Theorem \ref{main thm linear part} for $\alpha=1$] 
Since the function $h_t$ is radial, by Theorem \ref{Youngs ineq}, we have
\begin{equation}\label{Youngs ineq for heat ker}
    \|f \ast_k h_t\|_{L^r(d\mu_k)} \leq \|f\|_{L^p(d\mu_k)} \|h_t\|_{L^q(d\mu_k)},
\end{equation}
for any function $f\in L^p(d\mu_k)$ and for any $1 \leq p, q, r \leq \infty$, with the condition that ${1}/{p} + {1}/{q} = 1 + {1}/{r}.$
\par Now, with Proposition \ref{1st moment prop}, Young’s inequality \eqref{Youngs ineq for heat ker}, and the denseness of smooth functions with compact support in $L^1(d\mu_k)$ in hand, the remainder of the proof follows by repeating the exact same lines as in the classical case and is therefore omitted. We refer the reader to the references \cite[p. 20]{VazquezABMFTHE} and \cite[p. 7]{VazquezABFTFHEITES} for the details in the classical case.
\end{proof}
\subsection{Proof for the case $0<\alpha <1$}
Finally, by utilizing the already proven result for $\alpha = 1$, we can now complete the proof of Theorem \ref{main thm linear part} for any $0 < \alpha < 1$ as follows.

\begin{proof}[Proof of Theorem \ref{main thm linear part} for $0<\alpha < 1$] By applying the subordination formula and Minkowski's integral inequality, the $L^p$ norm transfers from the fractional heat case to the usual heat case as follows.
\begin{eqnarray*}
&& t^{d_k/(2\,\alpha\,p')}\|h_{t,\alpha}\ast_k u_0 - Mh_{t,\alpha}\|_{L^p(d\mu_k)}\\
&\leq &t^{d_k/(2\,\alpha\,p')}\int_0^\infty \|h_{s}\ast_k u_0 - Mh_{s}\|_{L^p(d\mu_k)}\,  \eta_{t,\alpha}(s) \,ds\\
&= &t^{d_k/(2\,\alpha\,p')}\int_0^{t} \|h_{s}\ast_k u_0 - Mh_{s}\|_{L^p(d\mu_k)}\,  \eta_{t,\alpha}(s) \,ds \\
&& + t^{d_k/(2\,\alpha\,p')}\int_{t}^\infty \|h_{s}\ast_k u_0 - Mh_{s}\|_{L^p(d\mu_k)}\,  \eta_{t,\alpha}(s) \,ds\\
&=: & I_t + J_t.
\end{eqnarray*}
Keeping \eqref{Youngs ineq for heat ker}, \eqref{int of heat ker is 1}, and the properties of the density function $\eta_{t, \alpha}$ in mind and taking $t > 1$, we first obtain the estimate
\begin{eqnarray*}
 I_t &\leq&  t^{d_k/(2\,\alpha\,p')}\int_0^{t} \|h_{s}\|_{L^p(d\mu_k)}\big[\| u_0\|_{L^1(d\mu_k)} + |M|\big]\,  \eta_{t,\alpha}(s) \,ds \\
 &\lesssim & t^{d_k/(2\,\alpha\,p')}\int_0^{t} s^{-d_k/(2\,p')} \, \eta_{t,\alpha}(s) \,ds\\
 &\lesssim & t^{d_k/(2\,\alpha\,p')}\int_0^{t} s^{-d_k/(2\,p')} \, t^{\,a_\alpha}(1+s^{b_\alpha})\,e^{-c_\alpha\, t^{\frac{1}{1-\alpha}}s^{-\frac{\alpha}{1-\alpha}}}\, ds. 
 \end{eqnarray*}
We next perform a sequence of changes of variables. Since the explicit exponents of $s$ and $t$ are not required for the estimate, we do not track them explicitly. Instead, we denote these exponents by real numbers
$a_j = a_j(\alpha), b_j = b_j(\alpha), j = 1,2,3,4.$ First, we replace $t^{\frac{\alpha}{1-\alpha}} s^{-\frac{\alpha}{1-\alpha}}$ by $s$. We then apply a further change of variables, replacing $c_\alpha t s$ by $s$. After these steps, we arrive at
 \begin{eqnarray*}
 I_t &\lesssim& t^{a_1}\int_1^\infty s^{b_1} (1+ t^{a_2}s^{b_2})\, e^{-c_\alpha t\,s}\, ds\\
 &\lesssim & t^{a_3}\int_{c_\alpha  t}^\infty s^{b_3} (1+ t^{a_4}s^{b_4})\, e^{-s}\, ds,
 \end{eqnarray*}
which goes to $0$ for large $t$.
\par Now for any positive $\varepsilon$, by the case $\alpha = 1$, there exists a sufficiently large $T > 0$ such that
$$s^{d_k/(2\,p')} \|h_{s}\ast_k u_0 - Mh_{s}\|_{L^p(d\mu_k)}<\varepsilon,$$
for any $s$ larger than $T$.
\par Hence, for the term $J_t$, we take $t > T$ and apply the above result to derive the required decay as follows.
\begin{eqnarray*}
J_t &<&  \varepsilon\,t^{d_k/(2\,\alpha\,p')}\int_{t}^\infty  s^{-d_k/(2\,p')}\,  \eta_{t,\alpha}(s) \,ds \\  
&\leq & \varepsilon\,\int_{t^{1-1/\alpha}}^{\infty} s^{-d_k/(2\,p')}\,\eta_{1,\alpha}(s) \,ds 
\leq \varepsilon\,\int_{0}^{\infty} s^{-d_k/(2\,p')}\,\eta_{1,\alpha}(s) \,ds
\lesssim  \varepsilon.
\end{eqnarray*}
This concludes the proof of the theorem.
\end{proof}


\section{Proof of Theorem \ref{main thm non-linear part}}\label{Sec- Nonlinear Main Thm proof}
In this section, we first establish  global well-posedness and a comparison principle for the nonlinear equation \eqref{non-linear frac Dunkl heat eq}, and then prove the asymptotic behavior of its solution.
\begin{lem}\label{lem:wellposedness}
Let $m>1$ and let $0\leq f\in L^1(d\mu_k)$. Then equation \eqref{non-linear frac Dunkl heat eq}, with $f$ in place of $u_0$, admits a unique global
nonnegative mild solution $v$ satisfying
$
v\in C([0,\infty);L^1(d\mu_k))
\cap L^\infty_{\mathrm{loc}}\bigl((0,\infty);L^\infty(d\mu_k)\bigr)
$.
Moreover,
\begin{equation}\label{eq:lemma51-mild}
v(x,t)=h_{t,\alpha}*_k f(x)
-\int_0^t h_{t-s,\alpha}*_k v^m(\cdot,s)(x)\,ds.
\end{equation}
Here, the regularity conditions on $v$ imply that $v^m\in L^1\big((0,T);L^1(d\mu_k)\big)$ for every $T>0$, and the above formula holds in $L^1(d\mu_k)$, with the time integral understood as a Bochner integral. Furthermore, the solutions preserve the ordering of the initial
data: if $0\leq f\leq g$ and $v_f,v_g$ are the corresponding
solutions, then
\begin{equation}\label{eq:lemma51-order}
0\leq v_f(x,t)\leq v_g(x,t),
\end{equation}
for $t>0$ and for almost every $x\in\mathbb R^d$.
\end{lem}

\begin{proof}
We start with the case $0<\alpha<1$. Let $1\leq q \leq \infty$ and $\phi \in L^q(d\mu_k)  $. Then \eqref{int of heat ker is 1}, \eqref{eq:sub or for}, and Theorem \ref{Youngs ineq} give
\begin{equation*}
\|h_{t,\alpha}*_k\phi\|_{L^q(d\mu_k)}\leq\|\phi\|_{L^q(d\mu_k)}.
\end{equation*}
The subordination formula \eqref{eq:sub or for} and the fact that
$\eta_{t,\alpha}(s)\geq 0$ implies that $h_{t,\alpha}\geq 0$ and
$h_{t,\alpha}$ is radial. Therefore, by the positivity of Dunkl
translations acting on nonnegative radial functions,
$\tau_x^k h_{t,\alpha}(-y)\geq 0$ for all $x,y\in\mathbb{R}^d$.
Hence,
$\{e^{-t\,(-\Delta_k)^{\alpha}}\}_{t\geq0}$ is positivity
preserving. Once again, the properties of the subordinator, together with the strong continuity of the Dunkl heat semigroup on $L^1(d\mu_k)$, imply the strong continuity of the fractional Dunkl heat semigroup. Consequently, $\{e^{-t\,(-\Delta_k)^{\alpha}}\}_{t\geq0}$ is a strongly continuous
positive contraction semigroup on $L^1(d\mu_k)$, and a contraction on $L^\infty(d\mu_k)$. Also, it follows from the definition of the semigroup, together with the identity $E_k(ix,0)=1$, that the semigroup preserves integrals, that is,
\begin{equation}\label{eq preserving}
\int_{\mathbb R^d}e^{-t(-\Delta_k)^\alpha}f(x)\,d\mu_k(x)
=
\int_{\mathbb R^d}f(x)\,d\mu_k(x).
\end{equation}
In the case $\alpha=1$, the above required properties follow directly from those of the Dunkl heat semigroup.

We now consider $0\leq f\in L^1(d\mu_k)\cap L^\infty(d\mu_k)$.
Choose $R>\|f\|_{L^\infty(d\mu_k)}$ and define
$$
b_R(r)=\bigl(\min\{r_+,R\}\bigr)^m,
\qquad r\in\mathbb R.
$$
The function $b_R$ is globally Lipschitz, nondecreasing,
and satisfies $b_R(0)=0$. Consequently, the standard
contraction-mapping argument on $C([0,T];L^1(d\mu_k))$,
first for sufficiently small $T$ and then by iteration,
gives a unique global solution of
$$
v_f(t)=e^{-t(-\Delta_k)^\alpha}f
-\int_0^t e^{-(t-s)(-\Delta_k)^\alpha}
b_R(v_f(s))\,ds.
$$
Here the composition operator induced by $b_R$ is globally
Lipschitz on $L^1(d\mu_k)$.

For completeness, comparison for this truncated equation
follows by choosing $L\geq \operatorname{Lip}(b_R)$ and
rewriting the mild formula as
$$
\begin{aligned}
v_f(t)
={}&e^{-Lt}e^{-t(-\Delta_k)^\alpha}f\\
&+\int_0^t e^{-L(t-s)}e^{-(t-s)(-\Delta_k)^\alpha}
\bigl(Lv_f(s)-b_R(v_f(s))\bigr)\,ds.
\end{aligned}
$$
The function $r\mapsto Lr-b_R(r)$ is nondecreasing.
Thus positivity of the semigroup and Picard iteration
give comparison for ordered initial data. Comparison
with the zero solution yields $v_f\geq0$. Since $b_R\geq0$ and $e^{-t\,(-\Delta_k)^{\alpha}}$ is positivity preserving,
the original mild formula then gives
$$
0\leq v_f(t)\leq e^{-t(-\Delta_k)^\alpha}f.
$$
In particular, by $L^\infty$-contractivity,
$$
\|v_f(t)\|_{L^\infty(d\mu_k)}
\leq \|f\|_{L^\infty(d\mu_k)}<R.
$$
Hence $b_R(v_f)=v_f^m$, so $v_f$ is a global nonnegative
mild solution of the original equation. Using a common
truncation level for two bounded initial data gives
comparison for the corresponding solutions. This follows the Lipschitz-approximation approach of
\cite[Lemmas 3.1 and 3.2]{ChenVeronWang}; the argument above
uses only the stated properties of the fractional Dunkl heat semigroup.

Since $v_f$ is bounded and continuous with values in
$L^1(d\mu_k)$, its nonlinear term belongs to
$L^1((0,T);L^1(d\mu_k))$ for every $T>0$.
Integrating the mild formula and using \eqref{eq preserving}
therefore gives the mass identity
$$
\|v_f(t)\|_{L^1(d\mu_k)}
+\int_0^t\|v_f(s)^m\|_{L^1(d\mu_k)}\,ds
=\|f\|_{L^1(d\mu_k)}.
$$
If $0\leq f\leq g$ are bounded and integrable, comparison
and subtraction of these identities yield
$$
\begin{aligned}
&\|v_g(t)-v_f(t)\|_{L^1(d\mu_k)}
+\int_0^t
\|v_g(s)^m-v_f(s)^m\|_{L^1(d\mu_k)}\,ds\\
&\hspace{4cm}=\|g-f\|_{L^1(d\mu_k)}.
\end{aligned}
$$
For arbitrary nonnegative bounded integrable $f$ and $g$,
comparison also gives
$$
(v_f(t)-v_g(t))_+
\leq v_{\max\{f,g\}}(t)-v_g(t).
$$
Applying the preceding estimate to $\max\{f,g\}$ and $g$,
we obtain
$$
\|(v_f(t)-v_g(t))_+\|_{L^1(d\mu_k)}
\leq \|(f-g)_+\|_{L^1(d\mu_k)}.
$$
Interchanging $f$ and $g$ and adding gives
$$
\|v_f(t)-v_g(t)\|_{L^1(d\mu_k)}
\leq \|f-g\|_{L^1(d\mu_k)}.
$$

Now let $0\leq f\in L^1(d\mu_k)$, and set
$f_n=\min\{f,n\}$ and $v_n=v_{f_n}$.
Comparison gives $v_n\leq v_{n+1}$, while contraction gives,
for every $T>0$,
$$
\sup_{0\leq t\leq T}
\|v_n(t)-v_j(t)\|_{L^1(d\mu_k)}
\leq \|f_n-f_j\|_{L^1(d\mu_k)}.
$$
Thus there exists
$v_f\in C([0,\infty);L^1(d\mu_k))$ such that
$v_n\to v_f$ in $C([0,T];L^1(d\mu_k))$ for every $T>0$,
and $v_n\to v_f$ almost everywhere.
Moreover, $v_f(0)=f$.

The mass identity gives
$$
\int_0^T\|v_n(s)^m\|_{L^1(d\mu_k)}\,ds
=\|f_n\|_{L^1(d\mu_k)}-\|v_n(T)\|_{L^1(d\mu_k)}
\leq \|f\|_{L^1(d\mu_k)}.
$$
Since $v_n^m\to v_f^m$ almost everywhere,
monotone convergence implies
$v_f^m\in L^1((0,T);L^1(d\mu_k))$ and $v_n^m\to v_f^m$ in $L^1((0,T);L^1(d\mu_k))$. By the $L^1$-contractivity of the semigroup,
$$
\begin{aligned}
&\sup_{0\leq t\leq T}
\left\|
\int_0^t e^{-(t-s)(-\Delta_k)^\alpha}
\bigl(v_n(s)^m-v_f(s)^m\bigr)\,ds
\right\|_{L^1(d\mu_k)}\\
&\qquad\leq
\int_0^T\|v_n(s)^m-v_f(s)^m\|_{L^1(d\mu_k)}\,ds
\rightarrow 0.
\end{aligned}
$$
Passing to the limit in the mild formulas therefore gives
$$
v_f(t)=e^{-t(-\Delta_k)^\alpha}f
-\int_0^t e^{-(t-s)(-\Delta_k)^\alpha}v_f(s)^m\,ds,
$$
which is \eqref{eq:lemma51-mild}.

Passing to the limit in
$0\leq v_n(t)\leq e^{-t(-\Delta_k)^\alpha}f_n$
also gives
$$
0\leq v_f(t)\leq e^{-t(-\Delta_k)^\alpha}f.
$$
Hence Theorem~\ref{Youngs ineq} and
\eqref{Lp norm of frac heat ker} yield
$$
\|v_f(t)\|_{L^\infty(d\mu_k)}
\leq\|h_{t,\alpha}*_k f\|_{L^\infty(d\mu_k)}
\lesssim t^{-d_k/(2\alpha)}\|f\|_{L^1(d\mu_k)}.
$$
In particular,
$v_f\in L^\infty_{\mathrm{loc}}
((0,\infty);L^\infty(d\mu_k))$.

Finally, let $v_f$ and $v_g$ be any two nonnegative mild
solutions in the stated class, with initial data $f$ and $g$.
Fix $T>0$. On each compact time interval contained in
$(0,T]$, the solutions are essentially bounded.
For almost every $\varepsilon\in(0,T)$, their values at
time $\varepsilon$ belong to
$L^1(d\mu_k)\cap L^\infty(d\mu_k)$.
Choose a common truncation level larger than the essential
bounds of both solutions on $[\varepsilon,T]$ and their
$L^\infty(d\mu_k)$ norms at time $\varepsilon$.
Restarting their mild formulas at time $\varepsilon$,
both solutions satisfy the corresponding truncated equation.
Uniqueness for this globally Lipschitz equation identifies
them with the bounded-data solutions constructed above.
Consequently, the bounded-data comparison and contraction
estimates give
$$
\|(v_f(t)-v_g(t))_+\|_{L^1(d\mu_k)}
\leq
\|(v_f(\varepsilon)-v_g(\varepsilon))_+\|_{L^1(d\mu_k)},
\qquad \varepsilon\leq t\leq T.
$$
Letting $\varepsilon\to 0$ through these times and
using continuity in $L^1(d\mu_k)$, we obtain
$$
\|(v_f(t)-v_g(t))_+\|_{L^1(d\mu_k)}
\leq \|(f-g)_+\|_{L^1(d\mu_k)}.
$$
Taking $f=g$ and interchanging the solutions proves
uniqueness, while taking $f\leq g$ proves
\eqref{eq:lemma51-order}.
\end{proof}

\begin{proof}[Proof of Theorem \ref{main thm non-linear part}]
Let $u$ be the unique global nonnegative mild solution
of \eqref{non-linear frac Dunkl heat eq} provided by Lemma \ref{lem:wellposedness}. Then
\begin{equation}\label{double int become zero}
    \int_{\mathbb{R}^d}u_0(x)\, d\mu_k(x)= \int_{\mathbb{R}^d}u(x, t)\, d\mu_k(x) + \int_0^t\int_{\mathbb{R}^d}u^m(x, s)\, d\mu_k(x)\,ds.
\end{equation}

Therefore, the condition on $u_0$ implies that, for almost every $t \in (0, \infty)$, we have $u(\cdot, t) \in L^1(d\mu_k)$, and moreover, $u \in L^m(d\mu_k \, dt)$.

\par Now, from the integral representation \eqref{eq:lemma51-mild} of the solution to \eqref{non-linear frac Dunkl heat eq}, the solution $u$ can be expressed as
\begin{equation}\label{eq: int rep formula}
 u(x,t) = h_{t,\alpha}\ast_k u_0(x) - \int_0^t h_{t-s,\alpha}\ast_k u^m(x, s)\, ds.   
\end{equation}
Next, we note that the fractional Dunkl heat kernel
$h_{t,\alpha}$ is radial and nonnegative for every $t>0$. Indeed, by the
subordination formula \eqref{eq:sub or for} and the fact that
$\eta_{t,\alpha}(s)\geq 0$, it follows that $h_{t,\alpha}\geq 0$ and that
$h_{t,\alpha}$ is radial. Therefore, by the positivity of Dunkl
translations acting on nonnegative radial functions,
$\tau_x^k h_{t,\alpha}(-y)\geq 0$ for all $x,y\in\mathbb{R}^d$.

Since the Dunkl convolution is commutative, for $0<s<t$ we have
$$
\begin{aligned}
\bigl(h_{t-s,\alpha} *_k u^m(\cdot,s)\bigr)(x)
&=
\int_{\mathbb{R}^d}
u(y,s)^m\,
\tau_x^k h_{t-s,\alpha}(-y)\,d\mu_k(y).
\end{aligned}
$$
As $u\geq 0$, it follows that
$$
\bigl(h_{t-s,\alpha} *_k u^m(\cdot,s)\bigr)(x)\geq 0
$$
for almost every $x\in\mathbb{R}^d$ and $s\in(0,t)$. Hence,
$$
\int_0^t
\bigl(h_{t-s,\alpha} *_k u^m(\cdot,s)\bigr)(x)\,ds
\geq 0.
$$
Thus, using this inequality in \eqref{eq: int rep formula}, we obtain that, for any nonnegative solution $u$ of \eqref{non-linear frac Dunkl heat eq}, the following inequality holds almost everywhere:
\begin{equation}\label{nonlinear sol domi by lin sol}
    u(x, t) \leq h_{t,\alpha}\ast_k u_0(x) .
\end{equation}

To prove that $M_\infty>0$, it suffices to consider nonzero
nonnegative initial data in $L^1(d\mu_k)\cap L^m(d\mu_k)$.
Indeed, by \eqref{double int become zero}, we have
$u(\cdot,t_0)\in L^m(d\mu_k)$ for almost every $t_0>0$.
Moreover, the continuity of $u$ in $L^1(d\mu_k)$ at $t=0$,
established in Lemma \ref{lem:wellposedness}, ensures that
$\|u(\cdot,t_0)\|_{L^1(d\mu_k)}>0$ for all sufficiently
small $t_0>0$, since $u_0\ne0$.
We may therefore choose $t_0>0$ such that
$$
0\ne u(\cdot,t_0)\in L^1(d\mu_k)\cap L^m(d\mu_k).
$$
The function $v(\cdot,s):=u(\cdot,t_0+s)$ is the mild
solution of the same equation with initial datum
$v(\cdot,0)=u(\cdot,t_0)$, and
$$
\lim_{s\to\infty}\int_{\mathbb R^d}v(x,s)\,d\mu_k(x)
=M_\infty.
$$
Thus, for the proof of positivity of $M_\infty$ below,
we may assume that $u_0\in L^1(d\mu_k)\cap L^m(d\mu_k)$.

Now, from the pointwise estimate \eqref{nonlinear sol domi by lin sol}, \eqref{Youngs ineq for heat ker}, and \eqref{Lp norm of frac heat ker}, we can derive an $L^m$ estimate  
\begin{eqnarray*}
\|u(\cdot, t) \|_{L^m(d\mu_k)} &\leq & \|h_{t,\alpha}\ast_k u_0 \|_{L^m(d\mu_k)}\\
&\lesssim & \min \{ t^{-d_k/(2\,\alpha\,m')} \|u_0\|_{L^1(d\mu_k)},\|u_0\|_{L^m(d\mu_k)} \}\\
&=:& A(t, m, \alpha, u_0).
\end{eqnarray*}
Let us fix $\varepsilon \in (0, 1]$ and consider the nonnegative solution $u^\varepsilon$ to the nonlinear equation \eqref{non-linear frac Dunkl heat eq}, but with the initial condition $\varepsilon\, u_0$ in place of $u_0$. Then, by Lemma \ref{lem:wellposedness}, we get the comparison
$$u^\varepsilon(x, t) \leq u(x,t).$$
Let $M^\varepsilon_\infty$ denote the quantity 
$$M^\varepsilon_\infty = \lim\limits_{t\to \infty}\int_{\mathbb{R}^d}u^\varepsilon (x, t)\,d\mu_k(x).$$
We claim that for a sufficiently small $\varepsilon$, $M^\varepsilon_\infty >0$.

\par In fact, using the relation \eqref{double int become zero}, we can express $M^\varepsilon_\infty$ as 
\begin{equation}\label{M infinity clacu}
  M^\varepsilon_\infty = \varepsilon \big[ \int_{\mathbb{R}^d}u_0(x)\, d\mu_k(x) - \frac{1}{\varepsilon}\int_0^\infty \int_{\mathbb{R}^d}\big(u^\varepsilon(x, t)\big)^m\, d\mu_k(x)\,dt \big].  
\end{equation}
Let us carefully look at the second term on the RHS of the above equation. The condition on $m$ gives
\begin{eqnarray*}
  &&\frac{1}{\varepsilon}\int_0^\infty \int_{\mathbb{R}^d}\big(u^\varepsilon(x, t)\big)^m\, d\mu_k(x)\,dt\\
  &\leq &  \frac{1}{\varepsilon}\int_0^\infty \big( A(t, m, \alpha, \varepsilon\,u_0)\big)^m\,dt\\
  &=&  \frac{\varepsilon^m}{\varepsilon}\int_0^\infty \big( A(t, m, \alpha, u_0)\big)^m\, dt\\
  &\leq & \varepsilon^{m-1}\big[\int_0^1 \|u_0\|^m_{L^m(d\mu_k)} \,dt  + \int_1^\infty  t^{-d_k(m-1)/(2\,\alpha)} \|u_0\|^m_{L^1(d\mu_k)}\, dt\big]\\
  &\lesssim &\varepsilon^{m-1},
\end{eqnarray*}
which goes to $0$ as $\varepsilon \to 0^+$.
\par Hence, from \eqref{M infinity clacu}, we can conclude that $M^\varepsilon_\infty > 0$ for some sufficiently small $\varepsilon$. As $u^\varepsilon(x, t) \leq u(x, t)$, this completes the proof of the first part of the theorem.
\par For the second part, using the integral representation of the solution, along with Theorem \ref{Youngs ineq} and \eqref{Lp norm of frac heat ker} for $m=1$, we obtain that for any $t \geq t_0 \geq 0$,
\begin{eqnarray*}
 && \|u(\cdot, t) - M_\infty h_{t,\alpha}\|_{L^1(d\mu_k)}\\
 &\leq & \| u(\cdot, t) - h_{t-t_0,\alpha}\ast_k u(\cdot, t_0) \|_{L^1(d\mu_k)} + \| h_{t-t_0,\alpha}\ast_k u(\cdot, t_0)- M(t_0)  h_{t-t_0,\alpha}\|_{L^1(d\mu_k)} \\
 && +\, \| M(t_0)  h_{t-t_0,\alpha}- M(t_0)  h_{t,\alpha}\|_{L^1(d\mu_k)} + \| M(t_0)  h_{t,\alpha} - M_\infty  h_{t,\alpha}\|_{L^1(d\mu_k)}\\
 &\lesssim & \int_{t_0}^t \|u(\cdot, s)\|^m_{L^m(d\mu_k)}\, ds + \| h_{t-t_0,\alpha}\ast_k u(\cdot, t_0)- M(t_0)  h_{t-t_0,\alpha}\|_{L^1(d\mu_k)} \\
 && +\,| M(t_0)|\, \| h_{t-t_0,\alpha} - h_{t-t_0,\alpha}*_k h_{t_0,\alpha} \|_{L^1(d\mu_k)} + | M(t_0) - M_\infty|.
\end{eqnarray*}
Now, taking the limit as $t \to \infty$ and applying Theorem \ref{main thm linear part} with $u(\cdot, t_0)$ and $h_{t_0,\alpha}$ in place of $u_0$, we derive
\begin{eqnarray*}
    \limsup_{t\to \infty}\|u(\cdot, t) - M_\infty h_{t,\alpha}\|_{L^1(d\mu_k)}&\lesssim \int_{t_0}^\infty \|u(\cdot, s)\|^m_{L^m(d\mu_k)}\, ds + | M(t_0) - M_\infty|
\end{eqnarray*}
and in the final step, taking $t_0 \to \infty$, we obtain the $L^1$ estimate 
$$\lim_{t\to \infty}\|u(\cdot, t) - M_\infty h_{t,\alpha}\|_{L^1(d\mu_k)}=0.$$
We now prove the theorem for $1 < q < \infty$. For $1 < q < r<\infty$, by H\"older's inequality, \eqref{nonlinear sol domi by lin sol},  Theorem \ref{Youngs ineq}, and \eqref{Lp norm of frac heat ker}, we have
\begin{eqnarray*}
&& \|u(\cdot, t) - M_\infty h_{t,\alpha}\|_{L^q(d\mu_k)} \\
& \leq & \|u(\cdot, t) - M_\infty h_{t,\alpha}\|^{\theta}_{L^1(d\mu_k)} \Big[\|u(\cdot, t)\|^{1-\theta}_{L^r(d\mu_k)} + \|M_\infty h_{t,\alpha}\|^{1-\theta}_{L^r(d\mu_k)}  \Big ]   \\
&\lesssim & \|u(\cdot, t) - M_\infty h_{t,\alpha}\|^{\theta}_{L^1(d\mu_k)} \Big[\|u_0\|^{1-\theta}_{L^1(d\mu_k)} + M_\infty ^{1-\theta}\Big ]\,t^{-d_k/(2\alpha q')},
\end{eqnarray*}
where $$\theta = \frac{r-q}{q(r-1)}, \text{ so that }0<\theta <1.$$
This concludes the proof of the theorem, as the $L^1$ case has already been proved.
\end{proof}

\subsection*{Acknowledgments} 
The author is supported by Institute Postdoctoral Fellowship from IIT Bombay. The author also thanks both referees for their thorough reviews, which improved the presentation of the paper.

\section*{Declarations}

\subsection*{Competing interests}
The author declares that he has no competing interests.

\bibliographystyle{abbrv}

\bibliography{biblio}

\end{document}